\documentclass[12pt, a4paper]{scrartcl}
\usepackage{amsmath}
\usepackage{amsthm}
\usepackage{amsfonts}
\usepackage{amssymb}
\usepackage[latin1]{inputenc}
\DeclareMathOperator{\SL}{SL}
\DeclareMathOperator{\N}{\mathbb{N}}
\DeclareMathOperator{\Z}{\mathbb{Z}}
\DeclareMathOperator{\R}{\mathbb{R}}
\DeclareMathOperator{\C}{\mathbb{C}}
\DeclareMathOperator{\Q}{\mathbb{Q}}
\renewcommand{\H}{\mathbb{H}}

\DeclareMathOperator{\cusp}{cusp}
\DeclareMathOperator{\new}{new}

	\newtheorem{Satz}{Satz}[section]
	\newtheorem{Theorem}[Satz]{Theorem}
	\newtheorem{Lemma}[Satz]{Lemma}
	\newtheorem{Proposition}[Satz]{Proposition} 
	\newtheorem{Corollary}[Satz]{Corollary}
	
	\theoremstyle{definition} 
	\newtheorem{Definition}[Satz]{Definition}

\date{}
\author{Markus Schwagenscheidt}
\title{Nonvanishing modulo $\ell$ of Fourier coefficients of Jacobi forms}

\begin{document} 

%	My work contributes to the theory of Jacobi forms, an area of active research. Methods from the theory of half-integral weight modular forms are carried over to the situation of Jacobi forms and are used to strengthen previous indivisibility results for special values of automorphic L-functions, which in turn have applications in the theory of elliptic curves. We study certain twisting, projection and Fricke operators on Jacobi forms, which, to the best knowledge of the author, have not been considered before. Moreover, an algebraicity result for the Fourier expansions at the cusps of Jacobi forms with algebraic coefficients is shown, which could have further applications.

	\maketitle

	\begin{abstract}
	Let $\phi = \sum_{r^{2} \leq 4mn}c(n,r)q^{n}\zeta^{r}$ be a Jacobi form of weight $k$ (with $k > 2$ if $\phi$ is not a cusp form) and index $m$ with integral algebraic coefficients which is an eigenfunction of all Hecke operators $T_{p}, (p,m) = 1,$ and which has at least one nonvanishing coefficient $c(n,r)$ with $r$ prime to $m$. We prove that for almost all primes $\ell$ there are infinitely many fundamental discriminants $D = r^{2}-4mn < 0$ prime to $m$ with $\nu_{\ell}(c(n,r)) = 0$, where $\nu_{\ell}$ denotes a continuation of the $\ell$-adic valuation on $\Q$ to an algebraic closure. As applications we show indivisibility results for special values of Dirichlet $L$-series and for the central critical values of twisted $L$-functions of even weight newforms.
	\end{abstract}

	\section{Introduction and main results}
	
	Let $k$ and $m$ be positive integers. Ono and Skinner \cite{OnoSkinner} and Bruinier \cite{BruinierNonvanishing} have shown, using different methods, that for a non-trivial Hecke eigenform $f = \sum_{n \geq 0}a(n)q^{n} \in M_{k+1/2}(4m,\chi)$ with integral algebraic coefficients, which is not a linear combination of Shimura theta functions, almost all primes $\ell$ have the property that there are infinitely many square free $n$ with $\nu_{\ell}(a(n)) = 0$, where $\nu_{\ell}:\overline{\Q} \to \Z \cup \{\infty\}$ denotes a continuation of the $\ell$-adic valuation on $\Q$ to a fixed algebraic closure.
	
	In this work, we adapt the methods from \cite{BruinierNonvanishing} to prove an analogous result for the Fourier coefficients of Jacobi forms:

	\begin{Theorem}\label{MainTheorem}
		Let $\phi \in J_{k,m}$ be a Jacobi form of weight $k$ (with $k > 2$ if $\phi$ is not a cusp form) and index $m$ with integral algebraic coefficients $c(n,r)$ which is an eigenform of all Hecke operators $T_{p}$ for $(p,m) = 1$. Suppose that there is a discriminant $D = r^{2}-4mn < 0$ with $(D,m) = 1$ and $c(n,r) \neq 0$. For almost all primes $\ell$ with $(\ell,2m) = 1$ there exist infinitely many fundamental discriminants $D = r^{2}-4mn < 0$ prime to $m$ such that $\nu_{\ell}(c(n,r)) = 0$.
	\end{Theorem}

	The finite set of primes $\ell$ for which the theorem does not hold can be described explicitly in terms of the Hecke eigenvalues of $\phi$, see (\ref{ExceptionalPrimes1}) and (\ref{ExceptionalPrimes2}). Further, if $S = \{p_{1},\dots,p_{s}\}$ is a finite set of odd primes which are coprime to $m$ and if $\varepsilon_{1},\dots,\varepsilon_{s} \in \{\pm 1\}$, then we can find (for allmost all primes $\ell$) infinitely many fundamental discriminants $D < 0$ as in the theorem satisfying the additional local conditions $\big( \frac{D}{p_{j}}\big) = \varepsilon_{j}$ for all $p_{j} \in S$. For a more detailed version of the theorem, which also covers Jacobi forms of higher level, see Theorem \ref{MainTheoremTechnical}.

	The condition $c(n,r) \neq 0$ for some $D = r^{2}-4mn$ prime to $m$ is obviously necessary, and is violated exactly by the forms in $\sum_{d^{2}\mid m, d > 1}J_{k,m/d^{2}}|U_{d}$ where $\phi|U_{d}(\tau,z) = \phi(\tau,dz)$, compare Lemma 3.1 in \cite{SkoruppaZagier}. In particular, the theorem is applicable for newforms.
	
%	Although the theorem is about Jacobi forms for the group $\SL_{2}(\Z) \ltimes \Z^{2}$, in the proof we have to work with Jacobi forms for subgroups of the form $\Gamma_{0}(N^{2})\ltimes (N\Z \times \Z)$ with a positive integer $N$. 
	
	The proof of the theorem goes along similar lines as the proof of the corresponding Theorem 4 from \cite{BruinierNonvanishing} for half-integral weight forms. The main ingredients are commutation relations of Hecke operators, Fricke involutions (with regard to the level), and twisting and projection operators for Jacobi forms on the group $\Gamma_{0}(N^{2})\ltimes (N\Z \times \Z)$, together with the $q$-expansion principle for integral weight modular forms. However, the aforementioned operators for Jacobi forms of higher level seem to be not very common, so it is interesting to see that similar arguments as for half-integral weight modular forms also work for Jacobi forms. 
	
	An important tool in the proof of Theorem \ref{MainTheorem} is the following result on the algebraicity of Fourier expansions of a Jacobi form at the cusps, which might be of independend interest:
	
	\begin{Theorem}\label{qExpansionPrinciple}
		Let $\phi \in J_{k,m}(\Gamma_{1}(N^{2}) \ltimes (N\Z \times \Z))$ have integral algebraic Fourier coefficients $c(n,r)$ and let $\ell$ be a prime with $(\ell,2mN) = 1$. For $M \in \SL_{2}(\Z)$ the function $\phi|_{k,m}M$ has algebraic Fourier coefficients and it holds $\nu_{\ell}(\phi|_{k,m}M) \geq \nu_{\ell}(\phi)$.
	\end{Theorem}
	
	Here we wrote $\nu_{\ell}(\phi) = \inf_{n,r}(c(n,r))$ for a Jacobi form $\phi = \sum_{r^{2}\leq 4mn}c(n,r)q^{n}\zeta^{r}$. For the proof, we apply the $q$-expansion principle to the components of the vector valued modular form corresponding to $\phi$ via its theta decomposition.

%	
%	\begin{Theorem}
%		Let $J_{k,m}^{\cusp,\new}$ be a Jacobi newform with integral algebraic coefficients which is an eigenfunction of all Hecke operators $T_{p}$ for $(p,m) = 1$. Let $\epsilon_{p}$ denote the eigenvalue of $\phi$ under the Atkin-Lehner involution $w_{p}$ for primes $p \mid \mid m$. Let $\ell$ be a prime with $(\ell,2m) = 1$, $\ell \notin A_{2m}(\phi), \ell \nmid (p-\epsilon_{p})$ for all primes $p \mid \mid m$ and $\nu_{\ell}(\phi) = 0$. Then there exist infinitely many fundamental discriminants $D = r^{2}-4mn$ prime to $m$ such that $\nu_{\ell}(c(n,r)) = 0$.
%	\end{Theorem}
	
	Let us now give two applications of Theorem \ref{MainTheorem}, both of which are strengthenings of applications from \cite{BruinierNonvanishing}: First, we consider the central critical values of twisted $L$-functions of even weight newforms. By the results of \cite{SkoruppaZagier}, the space of cuspidal Jacobi newforms $J_{k+1,m}^{\cusp,\new}$ of weight $k+1$ and index $m$ is isomorphic as a module over the Hecke algebra to the space $S_{2k}^{\new,-}(m)$ of elliptic newforms of weight $2k$ and level $m$ whose $L$-function $L(f,s)$ has a minus sign in its functional equation. Let $f \in S_{2k}^{\new,-}(m)$ be a normalised newform and let $\phi \in J_{k+1,m}^{\cusp,\text{new}}$ be the corresponding Jacobi newform with integral algebraic coefficients. If $D = r^{2}-4mn < 0$ with $(D,m) = 1$ is a fundamental discriminant which is a square mod $4m$, the formula of Waldspurger \cite{Waldspurger} in the explicit form given by Gross, Kohnen and   Zagier \cite{GKZ} states that
	\begin{align}\label{Waldspurger}
	\frac{|c(n,r)|^{2}}{\langle \phi,\phi\rangle} = \frac{(k-1)!}{2^{2k-1}\pi^{k}m^{k-1}}|D|^{k-1/2}\frac{L(f,D,k)}{\langle f,f\rangle},
	\end{align}
	where $L(f,D,s)$ denotes the $L$-function of $f$ twisted by the Kronecker symbol $\big(\frac{D}{\cdot}\big)$. We obtain from Theorem \ref{MainTheorem}:
	
	\begin{Corollary}\label{CorollaryWaldspurger}
		For a normalised newform $f$ in $S_{2k}^{\new,-}(m)$ there exists a nonzero complex number $\Omega_{f}$ such that for almost all primes $\ell$ with $(\ell,2m) = 1$ there are infinitely many fundamental discriminants $D < 0$ prime to $m$ which are squares mod $4m$ such that
		\[
		\nu_{\ell}\left( \frac{L(f,D,k)|D|^{k-1/2}}{\Omega_{f}}\right) = 0.
		\]
	\end{Corollary}
	
	Note that for a newform $f \in S_{2k}^{\new,+}(m)$ the functional equation of the twisted $L$-function implies $L(f,D,k) = 0$ for fundamental discriminants $D < 0$, $(D,m) = 1$, which are squares mod $4m$.  Compared to the very similar Theorem $5$ in \cite{BruinierNonvanishing} or Corollary 2 in \cite{OnoSkinner}, the discriminants $D$ in our corollary are now all coprime to $m$ and squares mod $4m$. 

	As explained in \cite{OnoSkinner}, Corollaries 3 and 4, a result like Corollary \ref{CorollaryWaldspurger} has implications for the orders of the Tate-Shafarevich groups of quadratic twists of elliptic curves over $\Q$. For related results, see also \cite{KohnenOno}.

	As a second application, we consider the coefficients of the Eisenstein series $E_{k,1} \in J_{k,1}$ for even $k \geq 4$, which are given by special values of Dirichlet $L$-series. $E_{k,1}$ is an Eigenform of all Hecke operators $T(p)$ with eigenvalue $\sigma_{2k-3}(p)$. By Theorem 2.1 in \cite{EichlerZagier}, the Fourier coefficients $e_{k,1}(n,r)$ of $E_{k,1}$ are $0$ or $1$ for $r^{2} = 4n$, and 
	\[
	e_{k,1}(n,r) = \frac{L_{D}(2-k)}{\zeta(3-2k)}, \quad \text{for } D = r^{2}-4n < 0,
	\]
	where
	\[
	L_{D}(s) = L_{D_{0}}(s)\sum_{d\mid f}\mu(d)\left( \frac{D_{0}}{d}\right)d^{-s}\sigma_{1-2s}(f/d)
	\]
	if $D = f^{2}D_{0}$ with a fundamental discriminant $D_{0} < 0$ and $f \in \N$, $L_{D_{0}}(s)$ is the Dirichlet series associated to $\chi_{D_{0}} = \big(\frac{D_{0}}{\cdot}\big)$ and $\zeta(s)$ is the Riemann $\zeta$-function. It is well known that $L_{D}(2-k) = -B_{k-1,\chi_{D}}/(k-1)$ with the generalised Bernoulli numbers $B_{k-1,\chi_{D}}$, so the values $L_{D}(2-k)$ are non-zero rational numbers whose denominators are divisible only by primes $\ell$ with $(\ell-1) \mid 2(k-1)$, see \cite{Carlitz}. Hence the coefficients of $\zeta(3-2k)E_{k,1}$ are $\ell$-integral rational numbers with bounded denominators for almost all primes $\ell$. The same is true for the coefficients of $\zeta(3-2k)E_{k,1}|V_{m}$, where $V_{m}:J_{k,1} \to J_{k,m}$ acts on Fourier expansions as
		\[
	V_{m}: \sum_{r^{2} \leq 4n}c(n,r)q^{n}\zeta^{r} \mapsto \sum_{r^{2}\leq 4mn}\bigg(\sum_{d\mid (n,r,m)}d^{k-1}c\left(\frac{nm}{d^{2}},\frac{r}{d}\right)\bigg)q^{n}\zeta^{r}.
	\]
	 Moreover, for any $\phi \in J_{k,1}$ and a fundamental discriminant $D = r^{2}-4mn< 0$ we have $c_{\phi|V_{m}}(n,r) = c_{\phi}(mn,r)$. Thus Theorem \ref{MainTheorem} applied to a suitable multiple of $\zeta(3-2k)E_{k,1}|V_{m}$ gives the following result:
	
	\begin{Corollary}
		Let $m$ be a positive integer and let $k \geq 4$ be an even integer. For almost all primes $\ell$ with $(\ell,2m) = 1$ there are infinitely many fundamental discriminants $D < 0$ prime to $m$ which are squares mod $4m$ such that $\nu_{\ell}(L_{D}(2-k)) = 0$.
	\end{Corollary}
	
	On the other hand, for odd $k$ and negative $D$ the functional equation of $L_{D}(s)$ gives $L_{D}(2-k) = 0$.
	
	Eventually, we remark that the techniques of this paper can be used to prove analogous results for skew-holomorphic Jacobi forms as defined in \cite{SkoruppaDevelopments}: Recall that a skew-holomorphic Jacobi form of weight $k$ and index $m$ is a smooth function $\phi: \H \times \C \to \C$ which is holomorphic in $z$, which is invariant under a modified slash action $|_{k,m}^{*}$ of $\SL_{2}(\Z)\ltimes \Z^{2}$ where $(c\tau + d)^{k}$ is replaced by $(c\overline{\tau}+d)^{k-1}|c\tau+d|$ (see below for the slash operator $|_{k,m}$ on holomorphic Jacobi forms), and which has a Fourier expansions of the form
	\[
	\phi(\tau,z) = \sum_{\substack{n,r \in \Z \\ r^{2} \geq 4mn}}c(n,r)e\left( \frac{r^{2}-4mn}{2m}iv\right)q^{n}\zeta^{r}, \qquad (\tau = u+iv).
	\]
	If we replace $r^{2}\leq 4mn$ by $r^{2}\geq 4mn$ and $|_{k,m}$ by $|_{k,m}^{*}$, and add a factor $e(\frac{r^{2}-4mn}{2m}iv)$ in the Fourier expansions in all proofs in Section $2$, we see that the results in Section $2$ are also valid for skew-holomorphic forms.
	
	Since a skew-holomorphic Jacobi form also has a theta decomposition and thereby corresponds to a vector valued modular form (see \cite{BruinierBorcherdsProducts}, Section 1.1), the proof of Theorem \ref{qExpansionPrinciple} also works for skew-holomorphic forms.
	
	Further, Theorem \ref{MainTheoremTechnical} and the propositions and lemmas used for its derivation have counterparts for skew-holomorphic Jacobi forms (just replace $D \leq 0$ with $D \geq 0$ everywhere) and can be proved in the same way. But we have to be careful with the remarks after Theorem \ref{MainTheoremTechnical}: It seems (to the best knowledge of the author) that the analog of the isomorphism between $J_{k,m}^{\cusp}$ with a subspace of $S_{2k-2}$ for skew-holomorphic Jacobi forms has not been proved in the literature, so estimates for the eigenvalues of skew-holomorphic Jacobi forms are not readily available.
	
	The situation is similar with the Waldspurger formula (\ref{Waldspurger}). Although it should be true for skew-holomorphic forms by similar arguments as in the holomorphic case, we could not find a proof in the literature. Assuming the formula, we get an analog of Corollary \ref{CorollaryWaldspurger} for newforms $f \in S_{2k}^{\new,+}(m)$, now yielding infinitely many positive fundamental discriminants $D = r^{2}-4mn > 0$ with $\nu_{\ell}(L(f,D,k)|D|^{k-1/2}/\Omega_{f}) = 0$.
	
%	\begin{Theorem}
%		Let $\phi \in J_{k,m}$ be a Jacobi form which is an eigenform of all the Hecke operators $T_{q}$ and whose Fourier coefficients $c(n,r)$ are algebraic and $|c(n,r)|$ depends only on the discriminant $D = r^{2}-4mn$ but not on the choice of $r \in \Z$ with $D \equiv r^{2} (4m)$. Then all but finitely many primes $\ell$ have the following property: If $\nu_{\ell}(\phi) = 0$ then there exist infinitely many fundamental discriminants $D < 0$ such that $\nu_{\ell}(c(D)) = 0$.
%	\end{Theorem}

%	, namely
%	\[
%	A_{m}(\phi) = \bigcap_{\substack{q \text{ prime} \\ (q,4m) = 1}}A'(q,\lambda_{q})
%	\]
%	where
%	\[
%	A'(q,\lambda_{q}) = \{\ell \text{ prime }: \nu_{\ell}(\lambda_{p}-\varepsilon(p^{k-1}+p^{k-2})) > 0 \text{ for } \varepsilon  = \pm 1\} \cup \{\ell \text{ prime }: \ell \mid 4mq(q-1)\}.
%	\]

		\section{Operators on Jacobi forms of higher level}
	
	For the basic facts about Jacobi forms (of level $1$) we refer the reader to the standard reference \cite{EichlerZagier}. In this section we define Hecke operators and the Fricke involution on Jacobi forms for $\Gamma_{0}(N^{2})\ltimes (N\Z \times \Z$), as well as certain twisting and projection operators, and we study their commutation relations.
	
	The real Jacobi group $\SL_{2}(\R) \ltimes \R^{2} \cdot S^{1}$ with group law
	\[
	[M,X,\zeta]\cdot[M',X',\zeta'] = \left[MM',XM' + X', \zeta\zeta'e\left(\det\binom{XM'}{X'}\right)\right]
	\]
	acts on holomorphic functions $\phi: \H \times \C \to \C$ by
	\begin{align*}
	&\phi\bigg|_{k,m}\left[\begin{pmatrix}a & b \\ c & d \end{pmatrix},[\lambda,\mu],\zeta \right](\tau,z)  \\
	&\qquad \qquad = \zeta^{m}(c\tau + d)^{-k}e^{m}\left(-\frac{c(z + \lambda \tau + \mu)^{2}}{c\tau + d} + \lambda^{2}\tau + 2 \lambda z + \lambda \mu\right) \\
	& \qquad \qquad  \times \phi\left( \frac{a\tau + b}{c\tau + d}, \frac{z+ \lambda \tau + \mu}{c\tau + d}\right),
	\end{align*}
	where $e(x):=e^{2\pi ix}$ and $e^{m}(x) := e^{2\pi i mx}$ for $x \in \C, m \in \Z$. To simplify the notation we sometimes write $\phi|_{k,m}M$ and $\phi|_{m}X$ for $\phi|_{k,m}[M,0,1]$ and $\phi|_{k,m}[1,X,1]$.

	\begin{Definition}
		Let $\Gamma \subseteq \SL_{2}(\Z)$ be a subgroup of finite index, let $\chi$ be a character of $\Gamma$ and let $L \subseteq \R^{2}$ be a rank 2 lattice which is invariant under right multiplication by $\Gamma$. A Jacobi form of weight $k$ and index $m$ for $\Gamma \ltimes L$ with character $\chi$ is a holomorphic function $\phi: \H \times \C \to \C$ with
		\begin{enumerate}
			\item $\phi|_{k,m}M = \chi(M)\phi$ for every $M \in \Gamma$.
			\item $\phi|_{m}X = \phi$ for every $X \in L$.
			\item For every $M \in \SL_{2}(\Z)$, the function $\phi|_{k,m}M$ has a Fourier expansion of the form
			\[
			\sum_{\substack{n,r \in \Q \\ r^{2} \leq 4mn}}c_{M}(n,r)q^{n}\zeta^{r}, \qquad (q^{n} = e(n\tau), \ \zeta^{r} = e(rz)),
			\]
			with coefficients $c_{M}(n,r) \in \C$.
		\end{enumerate}
		If for all $M \in \SL_{2}(\Z)$ we have $c_{M}(n,r) = 0$ whenever $r^{2}-4mn = 0$, we call $\phi$ a Jacobi cusp form. The corresponding spaces of Jacobi forms will be denoted by $J_{k,m}(\Gamma \ltimes L,\chi)$ and $J_{k,m}^{\cusp}(\Gamma\ltimes L,\chi)$, and for $\Gamma = \SL_{2}(\Z)$ and $L = \Z^{2}$ we just write $J_{k,m}$ and $J_{k,m}^{\cusp}$.
	\end{Definition}
	
	In the following, we will mainly work with Jacobi forms for the group $\Gamma_{0}(N^{2})\ltimes (N\Z \times \Z)$ and character $\chi(M) = \chi(d)$ for $M = \left(\begin{smallmatrix}a & b \\ c & d \end{smallmatrix}\right) \in \Gamma_{0}(N^{2})$, where $\chi$ is a Dirichlet character mod $N^{2}$. Note that the transformation behaviour of $\phi \in J_{k,m}(\Gamma_{0}(N^{2})\ltimes (N\Z \times \Z),\chi)$ under $N\Z \times \Z$ implies $c(n,r) = c(n + rN\lambda + mN^{2}\lambda^{2},r+2mN\lambda)$ for all $\lambda \in \Z$.
	
	Let $p$ be a prime. Similarly as in \cite{EichlerZagier}, we define the $p$-th Hecke operator on $J_{k,m}(\Gamma_{0}(N^{2})\ltimes(N\Z \times \Z),\chi)$ by
	\[
	\phi|T_{p} = p^{k-4}\sum_{\substack{M \in \Gamma_{0}(N^{2})\setminus M_{0}(N^{2}) \\ \det(M) = p^{2} \\ \text{gcd}(M) = \square}}\chi(a)\sum_{\lambda,\mu(p)}\phi\big|_{k,m}\left[\tfrac{1}{p}M, [N\lambda,\mu],1\right],
	\]
	where $M_{0}(N^{2})$ is the set of integral $2 \times 2$ matrices whose lower left entry is divisible by $N^{2}$, and $\text{gcd}(M) = \square$ means that the greatest common divisor of the entries of $M$ is a square.

	\begin{Lemma}
		The Hecke operator $T_{p}$ is an endomorphism of $J_{k,m}(\Gamma_{0}(N^{2})\ltimes (N\Z \times \Z),\chi)$ which maps cusp forms to cusp forms. The $(n,r)$-th coefficient of $\phi|T_{p} = \sum_{n,r}c^{*}(n,r)q^{n}\zeta^{r}$ is given by
		\begin{align*}
		c^{*}(n,r) &= c(p^{2}n,pr) + \chi(p)C_{p}(n,r,m)p^{k-3}c(n,r)  \\
		& \qquad + \chi(p^{2})p^{2k-3}\sum_{\lambda(p)}c((n + rN\lambda + mN^{2}\lambda^{2})/p^{2},(r+2mN\lambda)/p)
		\end{align*}
		where
		\[
		C_{p}(n,r,m) = \begin{cases}
		p\big( \frac{D}{p}\big), & \text{if } p \nmid m, \\
		0, & \text{if } p \mid m; \quad p \nmid r, \\
		-p, & \text{if } p \mid m, r; \quad  p \nmid n, \\
		p(p-1), & \text{if } p \mid m, r,n.
		\end{cases}
		\]
		Here we use the convention $c(n,r) = 0$ if $n$ or $r$ is not integral.
%		For $p \mid m$ the third part vanishes unless $p \mid m,r,n$.
	\end{Lemma}
	
	\begin{proof}
		The well-definedness and the mapping property of $T_{p}$ are easy exercises, so we only compute the Fourier expansion of $\phi|T_{p}$. As representatives for $\Gamma_{0}(N^{2}) \setminus M_{0}(N^{2})$ with determinant $p^{2}$ and square content we use the matrices
		\begin{align*}
		\alpha_{b} = \begin{pmatrix}1 & b \\ 0 & p^{2} \end{pmatrix}, \ 0 \leq b  < p^{2}, \qquad 		\beta_{h} = \begin{pmatrix}p & h \\ 0 & p \end{pmatrix}, \ 0 < h <  p, \qquad		\sigma = \begin{pmatrix}p^{2}& 0 \\0 & 1\end{pmatrix},
		\end{align*}
		(compare Theorem 1.7 in \cite{Shimura}). We have
		\[
		\left[\begin{pmatrix}1/p & b/p \\ 0 & p \end{pmatrix},[N\lambda,\mu],1\right] = \left[\begin{pmatrix}1 & 0 \\ 0 & 1 \end{pmatrix},[N\lambda p,0],1\right] \cdot \left[\begin{pmatrix}1/p & b/p \\ 0 & p \end{pmatrix},[0,\mu-N\lambda b],1\right],
		\]
		and thus the $\alpha_{b}$-part is given by
		\begin{align*}
		& p^{k-3}\sum_{b(p^{2})}\sum_{\mu(p)}\phi\bigg|_{k,m}\left[\begin{pmatrix}1/p & b/p \\ 0 & p \end{pmatrix},[0,\mu],1\right] \\
		&= p^{-3}\sum_{b(p^{2})}\sum_{\mu(p)}\sum_{\substack{n,r \in \Z \\ r^{2} \leq 4mn}}c(n,r)e(nb/p^{2})e(r\mu/p)q^{n/p^{2}}\zeta^{r/p} \\
		&= \sum_{\substack{n,r \in \Z \\ r^{2} \leq 4mn}}c(p^{2}n,pr)q^{n}\zeta^{r}.
		\end{align*}
		
		For $\beta_{h}$ we have
		\begin{align*}
		\left[ \begin{pmatrix}1 & h/p \\ 0 & 1 \end{pmatrix},[N\lambda,\mu],1\right] &= \left[ \begin{pmatrix}1 & 0 \\ 0 & 1 \end{pmatrix},[N\lambda,0],1\right]  \\
		&\quad  \cdot {} \left[ \begin{pmatrix}1 & h/p \\ 0 & 1 \end{pmatrix},[0,\mu-N\lambda h/p],e(N^{2}\lambda^{2}h/p)\right]. 
		\end{align*}
		The $\beta_{h}$-part gives $\chi(p)$ times
		\begin{align*}
		&p^{k-4}\sum_{h(p)^{*}}\sum_{\lambda,\mu (p)}\phi\bigg|_{k,m}\left[ \begin{pmatrix}1 & h/p \\ 0 & 1 \end{pmatrix},[0,\mu-N\lambda h/p],e(N^{2}\lambda^{2}h/p)\right] \\
		&= p^{k-3}\sum_{\substack{n,r \in \Z \\ r^{2}\leq 4mn}}c(n,r)\left(\sum_{h(p)^{*}}\sum_{\lambda(p)}e\left(\frac{h}{p}\left( n-rN\lambda + mN\lambda^{2} \right) \right)\right)q^{n}\zeta^{r}.
		\end{align*}
		The determination of the inner double sum is a standard computation involving quadratic Gauss sums, so we omit it for brevity.
		
		The $\sigma$-part is $\chi(p^{2})$ times
		\begin{align*}
		&p^{k-4}\sum_{\lambda,\mu (p)}\phi|_{k,m}[\sigma,[N\lambda,\mu],1] \\
		&= p^{2k-3}\sum_{\lambda(p)}\sum_{\substack{n,r \in \Z \\ r^{2}\leq 4mn}}c(n,r)q^{np^{2} + rpN\lambda + mN^{2}\lambda^{2}}\zeta^{rp + 2mN\lambda} \\
		&= p^{2k-3}\sum_{\substack{n,r \in \Z \\ r^{2}\leq 4mn}}\sum_{\lambda(p)}c((n + rN\lambda + mN^{2}\lambda^{2})/p^{2},(r+2mN\lambda)/p)q^{n}\zeta^{r}. 
		\end{align*}
		This completes the proof.
	\end{proof}
	
	Next, we define a Fricke involution on $J_{k,m}(\Gamma_{0}(N^{2})\ltimes (N\Z \times \Z),\chi)$.
	
	\begin{Lemma}
		Let $\phi \in J_{k,m}(\Gamma_{0}(N^{2}) \ltimes (N\Z \times \Z),\chi)$. Then
		\[
		\phi|W_{N^{2}}:=\phi\bigg|_{k,m} \begin{pmatrix}0 & -1/N \\ N & 0 \end{pmatrix} 
		\]
		is in $J_{k,m}(\Gamma_{0}(N^{2})\ltimes (N\Z \times \Z),\overline{\chi})$ and it holds $\phi |W_{N^{2}}|W_{N^{2}} = \chi(-1)\phi$. If $\phi$ is a cusp form, so is $\phi|W_{N^{2}}$.
	\end{Lemma}

	\begin{proof}
		For $\left(\begin{smallmatrix}a & b \\ c & d \end{smallmatrix}\right)  \in \Gamma_{0}(N^{2})$ and $[\lambda,\mu] \in N\Z \times \Z$ we have
		\begin{align*}
		&W_{N^{2}}\cdot \left[ \begin{pmatrix}a & b \\ c & d \end{pmatrix},[\lambda,\mu],1\right] =\left[ \begin{pmatrix}d & -c/N^{2} \\ -bN^{2} & a \end{pmatrix}, [-\mu N,\lambda/N],1\right] \cdot W_{N^{2}}
		\end{align*}
		The transformation behaviour now follows from $\chi(a) = \overline{\chi}(d)$ as $ad \equiv 1 \mod N^{2}$. The remaining properties are obvious.
	\end{proof}
	
	The existence of this nice Fricke involution is the main reason why we work with the group $\Gamma_{0}(N^{2}) \ltimes (N\Z \times \Z)$.
%	 rather than with $\Gamma_{0}(N^{2})\ltimes \Z^{2}$. 
%	One could try to define $W_{N^{2}}$ on $J_{k,m}(\Gamma_{0}(N^{2})\ltimes \Z^{2},\chi)$ by taking the average $\sum_{\lambda(N)}\phi|W_{N^{2}}|_{m}[\lambda,0]$, but this will in general (i.e. for $(m,N) > 1$) not be an involution. We later want to apply $W_{N^{2}}$ to a form $\phi$ of index $m$ and level $m$ whose coefficients $c(n,r)$ with $(r,m) = 1$ vanish, and in this case $\sum_{\lambda(N)}\phi|W_{N^{2}}|_{m}[\lambda,0] = \sum_{\lambda(N)}\phi|_{m}[0,-\lambda/N]|W_{N^{2}}$ is the zero operator.
	It is easy to check that for $(p,2mN) = 1$ we have the commutation relation
		\begin{align*}
		\phi|W_{N^{2}}|T_{p} &= \overline{\chi}(p^{2})\phi|T_{p}|W_{N^{2}}.
		\end{align*}

	We now introduce a twisting and a projection operator:
	
	\begin{Lemma}\label{TwistingProjectionOperators}
		Let $\phi = \sum_{r^{2}\leq 4mn}c(n,r)q^{n}\zeta^{r}\in J_{k,m}(\Gamma_{0}(N^{2})\ltimes (N\Z \times \Z),\chi)$. Further, let $p$ be a positive integer\footnote{Later, we will take $p$ to be prime, but the lemma is correct for arbitrary positive integers.}.
		\begin{enumerate}
			\item Let $\psi$ be a primitive Dirichlet character mod $p$. Then
			\[
			\phi_{\psi}(\tau,z) := \sum_{\substack{n,r \in \Z \\ r^{2}\leq 4mn} }\psi(r^{2}-4mn)c(n,r)q^{n}\zeta^{r}
			\]
			is in $J_{k,m}(\Gamma_{0}(N^{2}p^{2})\ltimes (Np\Z \times \Z),\chi\psi^{2})$.
			\item The function
			\[
			\phi|B_{p}(\tau,z) := \sum_{\substack{n,r \in \Z \\ r^{2}\leq 4mn \\ p \mid (r^{2}-4mn)} }c(n,r)q^{n}\zeta^{r}
			\]
			is in $J_{k,m}(\Gamma_{0}(N^{2}p^{2})\ltimes (Np\Z \times \Z),\chi)$.
		\end{enumerate}
		Both operators map cusp forms to cusp forms.
	\end{Lemma}
	
	\begin{proof}
		First note that the Fourier coefficients of $\phi_{\psi}$ and $\phi|B_{p}$ are invariant under $n \mapsto n + rN\lambda + mN^{2}\lambda^{2}, r\mapsto r+2mN\lambda$ for $\lambda \in \Z$, which is equivalent to the invariance of $\phi$ under $[\lambda,\mu] \in N\Z \times \Z$ (which is in fact stronger than invariance under $Np\Z \times \Z$, but we will not need this).

		We check the transformation rule of $\phi_{\psi}$ under $\Gamma_{0}(N^{2}p^{2})$: Since $\psi$ is primitive, we can write
		\[
		\psi(n) = \frac{1}{G(\overline{\psi})}\sum_{\alpha(p)^{*}}\overline{\psi}(\alpha)e(n\alpha/p)
		\]
		with the Gauss sum $G(\overline{\psi}) = \sum_{\alpha(p)^{*}}\overline{\psi}(\alpha)e(\alpha/p)$. We compute
		\begin{align*}
		\phi_{\psi}(\tau,z)&= \frac{1}{G(\overline{\psi})}\sum_{\alpha(p)^{*}}\overline{\psi}(\alpha)\sum_{\substack{n,r \in \Z \\ r^{2} \leq 4mn}}c(n,r)e((r^{2}-4mn)\alpha/p)q^{n}\zeta^{r} \\
		&= \frac{1}{G(\overline{\psi})}\sum_{\alpha(p)^{*}}\overline{\psi}(\alpha)\sum_{\rho (p)}\sum_{\substack{n,r \in \Z \\ r^{2} \leq 4mn \\ r \equiv \rho (p)}}c(n,r)e( (r^{2}-4mn)\alpha/p)q^{n}\zeta^{r} \\
		&= \frac{1}{G(\overline{\psi})}\sum_{\alpha(p)^{*}}\overline{\psi}(\alpha)\sum_{\rho (p)}e(\rho^{2}\alpha/p) \\
		& \qquad \times\sum_{\substack{n,r \in \Z \\ r^{2} \leq 4mn }}c(n,r)\left(\frac{1}{p}\sum_{\beta(p)}e((\rho - r)\beta/p)\right)e(-4mn\alpha/p))q^{n}\zeta^{r} \\
		&= \frac{1}{pG(\overline{\psi})}\sum_{\alpha(p)^{*}}\overline{\psi}(\alpha)\sum_{\beta(p)}G(\alpha,\beta,p)\phi(\tau,z)\bigg|_{m,k}\left[ \begin{pmatrix}1 & -4m\alpha/p \\ 0 & 1\end{pmatrix}, [0,-\beta/p],1\right],
		\end{align*}
		where $G(\alpha,\beta,p) = \sum_{\rho (p)}e((\rho^{2}\alpha+\rho \beta)/p)$ is a quadratic Gauss sum. For $M=\left( \begin{smallmatrix} a & b \\ c & d\end{smallmatrix}\right) \in \Gamma_{0}(N^{2}p^{2})$ we have
		\begin{align*}
		&\left[\begin{pmatrix}1 & -4m\alpha/p \\0 & 1 \end{pmatrix}, [0,-\beta/p],1 \right]\cdot \left[ M,[0,0],1\right]  \\
		&= \left[M',[\lambda',\mu'],1 \right] \cdot \left[ \begin{pmatrix}1 & -4md^{2} \alpha/p \\ 0 & 1 \end{pmatrix}, [0,-d\beta/p],1\right]
		\end{align*}
		where
		\[
		M' =\begin{pmatrix} a  -4m c\alpha /p & b +4mad^{2}\alpha /p -4md\alpha/p- 16m^{2}cd^{2}\alpha^{2}/p^{2} \\ c & d + 4m c\alpha d^{2} /p \end{pmatrix} \in \Gamma_{0}(N^{2}p^{2})
		\]
		and
		\[
		 [\lambda',\mu'] =  [-c\beta/p , 0]\begin{pmatrix}1 & 4md^{2} \alpha/p \\ 0 &  1\end{pmatrix} \in Np\Z \times \Z.
		\]
		Note that $\chi(M') = \chi(M)$. Hence we find
		\begin{align*}
		\phi_{\psi}|_{k,m}M &=\frac{1}{pG(\overline{\psi})}\chi(M)\sum_{\alpha(p)^{*}}\overline{\psi}(\alpha)\sum_{\beta(p)}G(\alpha,\beta,p)\phi(\tau - 4md^{2}\alpha/p,z - d\beta/p) \\
		&= \frac{1}{pG(\overline{\psi})}\chi(M)\sum_{\alpha(p)^{*}}\overline{\psi}\big(\overline{d}^{2}\alpha\big)\sum_{\beta(p)}G\big(\overline{d}^{2}\alpha,\overline{d}\beta,p\big)\phi(\tau -4m\alpha/p,z - \beta/p) \\
		&= \chi(M)\psi(d)^{2}\frac{1}{pG(\overline{\psi})}\sum_{\alpha(p)^{*}}\overline{\psi}(\alpha)\sum_{\beta(p)}G(\alpha,\beta,p)\phi(\tau -4m\alpha/p,z - \beta/p) \\
		&= \chi(M)\psi(M)^{2}\phi_{\psi}(\tau,z).
		\end{align*}
		This shows that $\phi_{\psi}$ transforms correctly. A standard argument shows that $\phi_{\psi}$ has suitable Fourier expansions at the cusps.
		
		For $\phi|B_{p}$ a similiar calculation as above gives
		\[
		\phi|B_{p}(\tau,z) = \frac{1}{p^{2}}\sum_{\alpha(p)}\sum_{\beta(p)}G(\alpha,\beta,p)\phi(\tau,z)\bigg|_{m,k}\left[ \begin{pmatrix}1 & -4m\alpha/p \\ 0 & 1\end{pmatrix}, [0,-\beta/p],1\right].
		\]
		Now we proceed in the same way as before to get the transformation law for $\phi|B_{p}$ under $\Gamma_{0}(N^{2}p^{2})$.
	\end{proof}
	
	For $p \mid m$ the proof actually shows that $\phi|B_{p}$ transforms as a Jacobi form for $\Gamma_{0}(N^{2}p)\ltimes (N\Z \times \Z)$. One could treat twists of Jacobi forms by $\chi(n)$ or $\chi(r)$ in a similar way. The latter, for example, has been considered in $\cite{MartinOsses}$.
	
	Let $\phi \in J_{k,m}(\Gamma_{0}(N^{2})\ltimes (N\Z \times \Z),\chi)$ and let $\psi$ be a quadratic character mod $p$. For each prime $q$ with $(q,p) = 1$ it holds
		\begin{align*}
		(\phi|T_{q})_{\psi} &= \phi_{\psi}|T_{q}, \\
		\phi|T_{q}|B_{p} &= \phi|B_{p}|T_{q}.
		\end{align*}
		This can be checked by comparing the actions of the operators on Fourier expansions. The interplay between the Fricke involution and the twisting and projection operators is described by the following lemma:

	\begin{Lemma}\label{CommutationRelations}
		Let $p$ be a prime with $(p,2mN) = 1$. For $\phi \in J_{k,m}(\Gamma_{0}(N^{2})\ltimes (N\Z \times \Z),\chi)$ and $\psi = \big( \frac{\cdot}{p}\big)$ we have
		\begin{align*}
		(\phi|W_{N^{2}}^{-1})_{\psi}|W_{N^{2}p^{2}} =  \overline{\chi}(p)p\sum_{\substack{n,r \in \Z \\ r^{2} \leq 4mn \\ p \mid n,r}}c(n,r)q^{n}\zeta^{r} -  \overline{\chi}(p)\sum_{\substack{n,r \in \Z \\ r^{2} \leq 4mn \\ p \mid r}}c(n,r)q^{n}\zeta^{r}
		\end{align*}
		and
		\begin{align*}
		\phi|W_{N^{2}}^{-1}|B_{p}|W_{N^{2}p^{2}} = p^{k-1}\phi(p^{2}\tau,pz) + \overline{\chi}(p)\sum_{\substack{n,r \in \Z \\ r^{2} \leq 4mn \\ p \mid r}}\left( \frac{r^{2}-4mn}{p}\right)c(n,r)q^{n}\zeta^{r}.
		\end{align*}
	\end{Lemma}

	\begin{proof}
		Using the representation of $\phi_{\psi}$ obtained in the proof of Lemma \ref{TwistingProjectionOperators} we have
		\begin{align*}
		(\phi|W_{N^{2}}^{-1})_{\psi}|W_{N^{2}p^{2}} &= \frac{1}{pG(\overline{\psi})}\psi(-m)\sum_{\alpha(p)^{*}}\overline{\psi}(\alpha)\sum_{\beta(p)}\left(\sum_{\rho(p)}e(-m(\rho^{2}\alpha + 2\rho \beta)/p) \right) \\
		& \qquad \times {} \phi\bigg|_{k,m}W_{N^{2}}^{-1}\left[\begin{pmatrix}1 & \alpha/p \\ 0 & 1 \end{pmatrix},[0,\beta/p],1 \right]  W_{N^{2}p^{2}},
		\end{align*}
		where we replaced $-4m\alpha$ by $\alpha$, $-\beta$ by $\beta$ and $\rho$ by $2m\rho$ mod $p$.
		Let $\overline{N^{2}\alpha} \in \Z$ denote an inverse of $N^{2}\alpha$ mod $p$. We have
		\begin{align*}
		W_{N^{2}}^{-1}\left[\begin{pmatrix}1 & \alpha/p \\ 0 & 1 \end{pmatrix},[0,\beta/p],1 \right] W_{N^{2}p^{2}} &=\left[\begin{pmatrix}p & \overline{N^{2}\alpha} \\ -N^{2}\alpha & (1-\overline{N^{2}\alpha}N^{2} \alpha)/p \end{pmatrix}, [N\beta,0],1\right] \\
		& \quad\times {} \left[\begin{pmatrix}1 & -\overline{N^{2}\alpha}/p \\ 0 & 1 \end{pmatrix}, [0,\overline{N^{2}\alpha}N\beta/p],e(-\overline{\alpha}\beta^{2}/p)\right].
		\end{align*}
		We obtain
		\begin{align*}
		(\phi|W_{N^{2}}^{-1})_{\psi}|W_{N^{2}p^{2}} &=\frac{\overline{\chi}(p)}{pG(\overline{\psi})}\psi(-m)\sum_{\alpha(p)^{*}}\overline{\psi}(\alpha)\sum_{\beta(p)}\left(\sum_{\rho (p)}e(-m(\rho^{2}\alpha+2\rho \beta + \overline{\alpha}\beta^{2})/p) \right) \\
		& \quad \times {} \phi\bigg|_{k,m}\left[\begin{pmatrix}1 & -\overline{N^{2}\alpha}/p \\ 0 & 1 \end{pmatrix},[0,\overline{N\alpha}\beta/p],1 \right]. \\
		\end{align*}
		The sum over $\rho$ is just $\psi(-m\alpha)G(\psi)$.
%		Further,
%		\[
%		\sum_{\rho (p)}e(-m(\rho^{2}\alpha+2\rho \beta+\overline{\alpha}\beta^{2})/p) = \sum_{\rho (p)}e(-m\alpha(\rho + \overline{\alpha}\beta)^{2}/p) = \psi(-m\alpha)G(\psi)
%		\]
%		where we used the Gauss sum relation $\sum_{\rho(p)}e(n\rho^{2}/p) = \sum_{\rho(p)^{*}}\big(\frac{\rho}{p}\big)e(n\rho/p)$. 
		Since $\psi = \overline{\psi}$, we find
		\[
		(\phi|W_{N^{2}}^{-1})_{\psi}|W_{N^{2}p^{2}} = \frac{\overline{\chi}(p)}{p}\sum_{\alpha(p)^{*}}\sum_{\beta(p)}\phi\bigg|_{k,m}\left[\begin{pmatrix}1 & \alpha/p \\ 0 & 1 \end{pmatrix},[0,\beta/p],1 \right].
		\]
		Plugging in the Fourier expansion of $\phi$ and an easy calculation give the stated formula.
		
		For the projection operator we use the representation for $\phi|B_{p}$ given at the end of the proof of Lemma \ref{TwistingProjectionOperators}. The summand for $\alpha = 0$ gives $p^{k-1}\phi(p^{2}\tau,pz)$ and the remaining part can be treated as above.
	\end{proof}

	\section{Proofs of the main results}
	
	We start with the proof of Theorem \ref{qExpansionPrinciple} as it will be needed in the proof of Theorem \ref{MainTheorem}.
	
%	\begin{Theorem}\label{qExpansionPrinciple}
%		Let $\phi \in J_{k,m}(\Gamma(N^{2}) \ltimes (N\Z \times \Z))$ have algebraic Fourier coefficients and let $\ell$ be a prime with $(\ell,N) = 1$ such that $\nu_{\ell}(\phi) \gg -\infty$. For $M \in \SL_{2}(\Z)$ the function $\phi|_{k,m}M$ has algebraic Fourier coefficients and it holds $\nu_{\ell}(\phi|_{k,m}M) \geq \nu_{\ell}(\phi)$.
%	\end{Theorem}
%	
	\begin{proof}[Proof of Theorem \ref{qExpansionPrinciple}]
		First note that $\phi|U_{N}(\tau,z) = \phi(\tau,Nz)$ is in $J_{k,mN^{2}}(\Gamma_{1}(N^{2}) \ltimes \Z^{2})$ and $\phi|U_{N}|_{k,mN^{2}}M = \phi|_{k,m}M|U_{N}$. Since $\phi|_{k,m}M$ and $\phi|_{k,m}M|U_{N}$ have the same set of Fourier coefficients, it suffices to show the claim for $\phi \in J_{k,mN^{2}}(\Gamma_{1}(N^{2})\ltimes \Z^{2})$. To simplify the notation, we take $\phi \in J_{k,m}(\Gamma_{1}(N)\ltimes \Z^{2})$.
		
		Such a form $\phi$ has a theta decomposition
		\[
		\phi(\tau,z) = \sum_{\mu(2m)}h_{\mu}(\tau)\theta_{m,\mu}(\tau,z)
		\]
		with
		\[
		h_{\mu}(\tau) = \sum_{D < 0}c_{\mu}(D)q^{-D/4m}, \qquad \theta_{m,\mu}(\tau,z) = \sum_{\substack{r \in \Z \\ r \equiv \mu (2m)}}q^{r^{2}/4m}\zeta^{r},
		\]
		where $c_{\mu}(D) = c((\mu^{2}-D)/4m,\mu)$, see  \cite{EichlerZagier}, \S 5. The theta function $\theta_{m,\mu}$ satisfies the transformation rules
		\begin{align*}
		\theta_{m,\mu}(\tau+1,z) &= e(\mu^{2}/4m)\theta_{m,\mu}, \\
		\theta_{m,\mu}\left(-\frac{1}{\tau},\frac{z}{\tau} \right) & = \sqrt{\tau/2mi}\ e(mz^{2}/\tau)\sum_{\nu(2m)}e(-\mu\nu/2m)\theta_{m,\nu}.
		\end{align*}
		This together with the transformation behaviour of $\phi$ under $\Gamma_{1}(N)$ implies that the tuple $(h_{\mu})_{\mu(2m)}$ transforms under $\Gamma_{1}(N)$ as a vector valued modular form of weight $k-1/2$ for the Weil representation $\rho_{L}$ of the lattice $L = \Z$ with quadratic form $q(x) = -mx^{2}$, compare \cite{BruinierBorcherdsProducts}, Section 1.1. For $M \in \SL_{2}(\Z)$ we have
		\[
		(\phi|_{k,m}M)(\tau,z) = \sum_{\mu(2m)}(h_{\mu}|_{k-1/2}M)(\tau) \cdot (\theta_{m,\mu}|_{1/2,m}M)(\tau,z).
		\]
		From the transformation rules of $\theta_{m,\mu}$ stated above it follows that $\theta_{m,\mu}|_{1/2,m}M$ is a linear combination $\theta_{m,\mu} = \sum_{\nu(2m)}\rho_{\nu,\mu}\theta_{m,\nu}$, where the coefficients $\rho_{\nu,\mu}$ are algebraic numbers with non-negative $\ell$-adic valuations. Here the assumption $(\ell,2mN) = 1$ was used. Thus to prove the claim of the theorem it remains to show that the functions $h_{\mu}|_{k-1/2}M$ have algebraic coefficients with $\nu_{\ell}(h_{\mu}|_{k-1/2}M) \geq \nu_{\ell}(h_{\mu})$.
		
		By the explicit formula for the action of $\Gamma_{1}(4m)$ in the Weil representation $\rho_{L}$ given in \cite{ScheithauerWeil}, Proposition 4.5., the functions $h_{\mu}$ are modular forms of weight $k-1/2$ for $\Gamma_{1}(4mN)$ and character $\chi_{\mu}\left(\left(\begin{smallmatrix}a & b \\ c & d \end{smallmatrix}\right)\right) = e(-b\mu^{2}/4m)$. We write
		\[
		h_{\mu}|_{k-1/2}M = \left(h_{\mu}\bigg|_{k-1/2}\begin{pmatrix}2m & 0 \\0 & \frac{1}{2m}\end{pmatrix}\right)\bigg|_{k-1/2}\begin{pmatrix} \frac{1}{2m} & 0 \\0 & 2m\end{pmatrix}M.
		\]
		The inner function is now a modular form of weight $k-1/2$ for $\Gamma_{1}(16m^{3}N)$ and trivial character with algebraic Fourier coefficients and the same $\ell$-adic valuation as $h_{\mu}$. The product of the remaining two matrices can be written as $M'\left( \begin{smallmatrix}\alpha & \beta \\ 0 & \delta \end{smallmatrix}\right)$ with $M' \in \SL_{2}(\Z)$ and $\alpha,\beta,\delta \in \Q$ with $\nu_{\ell}(\delta) = 0$. It now follows from the $q$-expansion principle for $\Gamma_{1}(N)$ (see \cite{Katz}, Corollary 1.6.2., or \cite{DiamondIm}, Remark 12.3.5., for the $q$-expansion principle for integral weight modular forms for $\Gamma_{1}(N)$, and the proof of Lemma 1 in \cite{BruinierNonvanishing} for the transition to half-integral weight forms) that $h_{\mu}|_{k-1/2}M$ has algebraic coefficients with $\nu_{\ell}(h_{\mu}|_{k-1/2}M) \geq \nu_{\ell}(h_{\mu})$. The assumption $(\ell,2mN) = 1$ is again needed here.
	\end{proof}

	\begin{Proposition}\label{OneDiscriminant}
		Let $p$ be a prime with $(p,2mN) = 1$ and let $\psi = \big(\frac{\cdot}{p}\big)$. Suppose that $\phi \in J_{k,m}(\Gamma_{0}(N^{2})\ltimes (N\Z \times \Z),\chi)$ has integral algebraic Fourier coefficients. Let $\ell$ be a prime with $(\ell,2mNp(p-1)) = 1$ and let
		\[
		h := \phi|(1-B_{p})-\varepsilon \phi_{\psi} = 2\sum_{\substack{n,r \in \Z  \\ \psi(r^{2}-4mn) = -\varepsilon}}c(n,r)q^{n}\zeta^{r}.
		\] 
		Then $\nu_{\ell}(\phi|T_{p}-\chi(p)\varepsilon(p^{k-1} + p^{k-2})\phi) \geq \nu_{\ell}(h)$. In particular, if $\phi$ is an eigenform of $T_{p}$ with eigenvalue $\lambda_{p}$ and $\nu_{\ell}(\phi) = 0$, there is a discriminant $D= r^{2}-4mn < 0$ such that $\big( \frac{D}{p}\big) = -\varepsilon$ and $\nu_{\ell}(c(n,r)) \leq \nu_{\ell}(\lambda_{p}-\chi(p)\varepsilon (p^{k-1}+p^{k-2}))$.
	\end{Proposition}

\begin{proof}
	Let $h' = h|W_{N^{2}p^{2}}$. If we write $\big(\begin{smallmatrix}0 & -1/Np \\ Np & 0 \end{smallmatrix}\big) = \big(\begin{smallmatrix}0 & - 1 \\ 1 & 0 \end{smallmatrix}\big)\big(\begin{smallmatrix}Np & 0 \\ 0 & 1/Np\end{smallmatrix}\big)$ Theorem \ref{qExpansionPrinciple} implies that $h'$ has algebraic coefficients with $\nu:=\nu_{\ell}(h) = \nu_{\ell}(h')$. We get equality instead of an inequality here since $W_{N^{2}p^{2}}$ is an involution (up to a sign). Further, let $c'(n,r)$ denote the Fourier coefficients of $\phi' =\phi|W_{N^{2}}$, which are again algebraic. We write
	\begin{align*}
	h '  &= \phi'|W_{N^{2}}^{-1}|W_{N^{2}p^{2}} - \phi'|W_{N^{2}}^{-1}|B_{p}|W_{N^{2}p^{2}} -  \varepsilon (\phi'|W_{N^{2}}^{-1})_{\psi}|W_{N^{2}p^{2}}.
	\end{align*}
	It is easily checked that $\phi'|W_{N^{2}}^{-1}|W_{N^{2}p^{2}}(\tau,z) = p^{k}\phi'(p^{2}\tau,pz)$. Since $\phi'$ transforms with character $\overline{\chi}$, we get from Lemma \ref{CommutationRelations} the formula:
	\begin{align*}
	h' &=  (p^{k}-p^{k-1})\sum_{\substack{n,r \in \Z \\ r^{2} \leq 4mn}}c'(n/p^{2},r/p)q^{n}\zeta^{r} - \chi(p)\sum_{\substack{n,r \in \Z \\ r^{2} \leq 4mn \\ p \mid r}}\left( \frac{r^{2}-4mn}{p}\right)c'(n,r)q^{n}\zeta^{r} \\
	&\quad  - \chi(p)\varepsilon p \sum_{\substack{n,r \in \Z \\ r^{2} \leq 4mn \\ p \mid n,r}}c'(n,r)q^{n}\zeta^{r} + \chi(p)\varepsilon \sum_{\substack{n,r \in \Z \\ r^{2} \leq 4mn \\ p \mid r}}c'(n,r)q^{n}\zeta^{r}.
	\end{align*}
	This gives the relations
	\begin{align}\label{DivisibilityRelations}
	\nu_{\ell}\left(\left(\frac{r^{2}-4mn}{p}\right)c'(n,r) - \varepsilon c'(n,r)\right) &\geq \nu, \quad \text{ if } p \mid r, \ p \nmid n, \\
	\nu_{\ell}(c'(n,r)) &\geq \nu, \quad \text{ if } p \mid r, \ p\mid n, \ p^{2} \nmid n, \\
	\nu_{\ell}(c'(n,r)-\overline{\chi}(p)\varepsilon p^{k-1}c'(n/p^{2},r/p)) &\geq \nu, \quad \text{ if } p \mid r, \ p^{2} \mid n,
	\end{align}
	where we divided by $\chi(p),\varepsilon,p^{k-1}$ or $p-1$ at some places. Using $c'(n,r) = c'(n + rN\lambda + mN^{2}\lambda^{2},r+2mN\lambda)$ for $\lambda \in \Z$ and the assumption $(p,2mN) = 1$ we may assume that $p \mid r$. For $p \mid r$ and $(p,2mN) = 1$, the coefficient $c'^{*}(n,r)$ of $\phi'|T_{p}$ is given by
	\begin{align*}
	c'^{*}(n,r) &= c'(p^{2}n,pr) + \overline{\chi}(p)\left( \frac{r^{2}-4mn}{p}\right)p^{k-2}c'(n,r) + \overline{\chi}(p^{2})p^{2k-3}c'(n/p^{2},r/p).
	\end{align*}
	Using the relations (2), (3), (4) it is straightfoward to verify that
	\[
	\nu_{\ell}(c'^{*}(n,r) - \overline{\chi}(p)\varepsilon(p^{k-1}+p^{k-2})c'(n,r)) \geq \nu.
	\]
	The result now follows from $\phi'|T_{p} = \phi|W_{N^{2}}^{-1}|T_{p} = \overline{\chi}(p^{2})\phi|T_{p}|W_{N^{2}}^{-1}$ and Theorem \ref{qExpansionPrinciple}. 
	\end{proof}

	\begin{Proposition}\label{InfinitelyDiscriminants}
		Let $p$ be a prime with $(p,2mN) = 1$ and let $\phi \in J_{k,m}(\Gamma_{0}(N^{2})\ltimes (N\Z \times \Z),\chi)$ with integral algebraic Fourier coefficients which is an eigenform of $T_{p}$ with eigenvalue $\lambda_{p}$. Let $\ell$ be a prime with $(\ell,2mNp(p-1)) = 1$ and let $\varepsilon \in \{\pm 1\}$. Suppose that there is some discriminant $D^{*} = (r^{*})^{2}-4mn^{*} < 0$ with $\big(\frac{D^{*}}{p}\big) = \varepsilon$ and $\nu_{\ell}(c(n^{*},r^{*})) = 0$. Then there exist infinitely many discriminants $D = r^{2}-4mn < 0$ with pairwise distinct square free parts and $\big(\frac{D}{p}\big) = -\varepsilon$ such that $\nu_{\ell}(c(n,r)) \leq \nu_{\ell}(\lambda_{p}-\chi(p)\varepsilon(p^{k-1}+p^{k-2}))$.
	\end{Proposition}

	\begin{proof}
		We prove by induction on $t$ that there are negative discriminants $D_{1} = r_{1}^{2}-4mn_{1}, \dots,D_{t} = r_{t}^{2}-4mn_{t}$ with pairwise distinct square free parts such that $\big(\frac{D_{j}}{p}\big) = -\varepsilon$ and $\nu_{\ell}(c(n_{j},r_{j})) \leq \nu_{\ell}(\lambda_{p}-\chi(p)\varepsilon(p^{k-1}+p^{k-2}))$ for $j = 1,\dots,t$.
		
		For $t = 1$ the statement follows from the last proposition.
		
		Now let us assume that for $\phi$ as above and fixed $t > 1$ we can always find $t-1$ discriminants with the desired properties. By the last proposition we find a discriminant $D_{1}= r_{1}^{2}-4mn_{1} < 0$ with $\big( \frac{D_{1}}{p}\big) = -\varepsilon$ and $\nu_{\ell}(c(n_{1},r_{1})) \leq \nu_{\ell}(\lambda_{p}-\chi(p)\varepsilon(p^{k-1}+p^{k-2}))$. 
%		We are now going to construct from $\phi$ a form $\phi_{1}$ of higher level whose coefficients $c_{1}(n,r)$ are $0$ at least for discriminants $D =r^{2}-4mn$ having the same square free part as $D_{1}$ and, if non-zero, are equal to $c(n,r)$, and then use the induction hypothesis on $\phi_{1}$ to find $D_{2},\dots,D_{t}$:
		Since $(\frac{D_{1}}{p}\big) = -\varepsilon$ and $\big(\frac{D^{*}}{p} \big) = \varepsilon$ the square free parts of $D_{1}$ and $D^{*}$ are different.
		Thus we can choose a prime $q$ with
		\[
		(q,\ell) = 1, \quad (q,p) = 1, \quad \left( \frac{D^{*}}{q}\right) = +\varepsilon, \quad \left(\frac{D_{1}}{q}\right) = -\varepsilon.
		\]
		Put $\psi = \big(\frac{\cdot}{q}\big)$ and consider the function
		\[
		\phi_{1} = \frac{1}{2}\left(\phi|(1-B_{q}) + \varepsilon \phi_{\psi} \right) \in J_{k,m}(\Gamma_{0}(N^{2}q^{2}) \ltimes (Nq\Z \times \Z),\chi\psi^{2}).
		\]
		Then $\phi_{1}$ has the Fourier expansion
		\[
		\phi_{1} = \sum_{\substack{n,r \in \Z\\ r^{2}\leq 4mn}}c_{1}(n,r)q^{n}\zeta^{r}, \quad \text{ with } c_{1}(n,r) = 			\begin{cases} 
		c(n,r) , & \text{if } \big(\frac{r^{2}-4mn}{q}\big) = \varepsilon, \\
		0, & \text{otherwise}.
		\end{cases}
		\]
		Since $T_{p}$ commutes with the twisting and projection operators in the definition of $\phi_{1}$, the function $\phi_{1}$ is also an eigenform of $T_{p}$ with eigenvalue $\lambda_{p}$. Further, $c_{1}(n^{*},r^{*}) = c(n^{*},r^{*})$, so $\nu_{\ell}(\phi_{1}) = \nu_{\ell}(\phi) = 0$. We can apply the induction hypothesis on $\phi_{1}$ to find discriminants $D_{2} = r_{2}^{2}-4mn_{2}, \dots , D_{t} = r_{t}^{2}-4mn_{t}$ with pairwise distinct square free parts such that $\big(\frac{D_{j}}{p}\big) = -\varepsilon$ and $\nu_{\ell}(c(n_{j},r_{j})) = \nu_{\ell}(c_{1}(n_{j},r_{j})) \leq \nu_{\ell}(\lambda_{p}-\chi(p)\varepsilon(p^{k-1}+p^{k-2}))$. Since $\big(\frac{D_{1}}{q}\big) = -\varepsilon$, we have $c_{1}(n,r) = 0$ for all $n,r$ such that $r^{2}-4mn$ has the same square free part as $D_{1}$, so the square free parts of $D_{2},\dots,D_{t}$ are different from the square free part of $D_{1}$. This completes the proof.
	\end{proof}
		
	So far, the discriminants given in Proposition \ref{InfinitelyDiscriminants} need not be fundamental. To find fundamental discriminants with the properties as in the proposition, we will make use of the relations given by the Hecke eigenform equations.
	
	\begin{Lemma}\label{HeckeEigenformRelation}
		Let $\phi \in J_{k,m}(\Gamma_{0}(N^{2}) \ltimes (N\Z \times \Z),\chi)$ be an eigenform of all Hecke operators $T_{p},(p,mN) = 1$, with algebraic Fourier coefficients $c(n,r)$. Let $\ell$ be a prime and $D = r^{2}-4mn \leq 0$ a fundamental discriminant. Then $\nu_{\ell}(c(f^{2}n,fr)) \geq \nu_{\ell}(c(n,r))$ for all $f \in \Z, (f,mN) = 1$.
	\end{Lemma}
	
	\begin{proof}
		By induction on $|f|$: For $f = 1$ there is nothing to show and for $f = -1$ we have $c(n,r) = (-1)^{k}c(n,-r)$. For $|f| > 1$ choose a prime divisor $p$ of $f$ and write $f = pf_{1}$. The action of $T_{p}$ on $\phi$ gives the relation
		\begin{align*}
		\lambda_{p}c(f_{1}^{2}n,f_{1}r) &= c(f^{2}n,fr) + \chi(p)\left(\frac{ f_{1}^{2}D}{p}\right)p^{k-2}c(f_{1}^{2}n,f_{1}r)  \\
		&\quad + \chi(p^{2})p^{2k-3}\sum_{\lambda(p)}c((f_{1}^{2}n+f_{1}rN\lambda + mN^{2}\lambda^{2})/p^{2}, (f_{1}r+2mN\lambda)/p).
		\end{align*}
		We find
		\begin{align*}
		\nu_{\ell}(c(f^{2}n,fr)) &\geq \min\bigg(\nu_{\ell}(\lambda_{q}c(f_{1}^{2}n,f_{1}r)),  \nu_{\ell}(p^{k-2}c(f_{1}^{2}n,f_{1}r)), \\
		& \qquad \nu_{\ell}\bigg(p^{2k-3}\sum_{\lambda(p)}c((f_{1}^{2}n+f_{1}rN\lambda + mN^{2}\lambda^{2})/p^{2}, (f_{1}r+2mN\lambda)/p)\bigg) \bigg).
		\end{align*}
		In the sum over $\lambda$ the summands can be nonzero only if $p \mid f_{1}$ since $D$ is fundamental, in which case the sum is just $c((f_{1}/p)^{2}n,(f_{1}/p)r)$ as $(p,mN) = 1$. Note that $\lambda_{p}$ is an algebraic integer, so all the terms in the minimum are at least $\nu_{\ell}(c(n,r))$ by the induction hypothesis, which proves the claim.
	\end{proof}
	
	We would now like to write each $D = r^{2}-4mn$ in Proposition \ref{InfinitelyDiscriminants} as $D = f^{2}D_{0}$ with a fundamental discriminant $D_{0} = r_{0}^{2}-4mn_{0}$ and $f \in \Z$, and then use the relations given by the eigenform equation to show that $\nu_{\ell}(c(n_{0},r_{0})) \leq \nu_{\ell}(c(n,r))$. But in general, $D_{0}$ will not necessarily be a square mod $4m$, as the example $m = 3$ and $D = 0^{2}-4\cdot 3 \cdot 3 = 9\cdot(-4)$ shows, and even if $D_{0}$ is a square mod $4m$, the eigenform equations will in general not give information about the relation of $c(n,r)$ and $c(f^{2}n_{0},fr_{0})$ if $r \not\equiv fr_{0} \mod 2m$. These problems can only arise if $(D,mN) > 1$:
	
	\begin{Lemma}\label{FundamentalDiscriminants}
		Let $\phi \in J_{k,m}(\Gamma_{0}(N^{2})\ltimes (N\Z \times \Z),\chi)$ with Fourier coefficients $c(n,r)$. Let $D = r^{2}-4mn< 0$ with $(D,mN) = 1$ and write $D = f^{2}D_{0}$ with a fundamental discriminant $D_{0} < 0$ and some $f \in \Z$. Then $D_{0}$ is a square mod $4m$, i.e. $D_{0} = r_{0}^{2} -4mn_{0}$, and $c(n,r) = c(f^{2}n_{0},f r_{0})$.
	\end{Lemma}
	
	\begin{proof}
		It suffices to show the claim in the case that $f = p$ is a prime with $(p,mN) = 1$ and $D_{0} <0 $ an arbitrary (i.e. not necessarily fundamental) discriminant.
		
		We will use that $c(n,r) = c(n+ rN\lambda + mN^{2}\lambda^{2},r+2mN\lambda)$ holds for every $\lambda \in \Z$.
	
		Assume that $p$ is odd. Since $2mN$ is invertible mod $p$ we can achieve $r + 2mN \lambda \equiv 0 (p)$, so we can assume that $p \mid r$. Write $r = pr_{0}$. Then $p^{2}r_{0}^{2}-4mn = p^{2}D_{0}$ implies $p^{2} \mid 4mn$, and $(p,2m) = 1$ shows $p^{2} \mid n$, i.e. $n = p^{2}n_{0}$.
		
		We leave the case $p = 2$ to the reader.
%		Now we look at the case $p = 2$. This is only possible if $mN$ is odd since $(p,mN) = 1$. From $r^{2}-4mn = 4D_{0}$ it follows that $2 \mid r$. If $r \equiv 2 (4)$ we replace $r$ by $r + 2mN \equiv 0 (4)$, i.e. we can assume $r \equiv 0 (4)$. Writing $D_{0} = r_{0}^{2}-4n_{0}$ we obtain $mn \equiv -(r_{0})^{2} (4)$ from $r^{2}-4mn = 4D_{0}$. If $2 \mid r_{0}$, we must have $4 \mid n$ since $m$ is odd. If $2 \nmid r_{0}$, we have $mn \equiv -1 (4)$, so one of $m,n$ is $1$ and the other one is $-1$ mod $4$. We replace $r$ by $r+2mN$ and $n$ by $n+rN + mN^{2} \equiv n + mN^{2} \equiv 0 (4)$. In each case we can assume $2 \mid r$ and $4 \mid n$, i.e. $D_{0} = r_{0}^{2} - 4mn_{0}$ with $r = 2r_{0}, n = 4n_{0}$.
 	\end{proof}
	
	If $\phi \in J_{k,m}(\Gamma_{0}(N^{2})\ltimes (N\Z \times \Z),\chi)$ is a Jacobi form with integral algebraic coefficients which is an eigenfunction of all Hecke operators $T_{p}$ for $(p,mN) = 1$ with corresponding eigenvalues $\lambda_{p}$, we define the set	
	\begin{align}\label{ExceptionalPrimes1}
	A(p,\lambda_{p}) &= \{\ell \text{ prime}: \nu_{\ell}(\lambda_{p}-\chi(p)\varepsilon (p^{k-1}+p^{k-2})) > 0 \text{ for some $\varepsilon  \in\{ \pm 1\}$}\} \notag \\
	&\qquad\qquad\qquad\qquad\qquad\qquad\qquad\qquad\qquad\cup\{\ell \text{ prime}: \ell \mid p(p-1)\}.
	\end{align}
	
	The following theorem is a more precise variant of Theorem \ref{MainTheorem} from the introduction.

	\begin{Theorem}\label{MainTheoremTechnical}
		Let $\phi \in J_{k,m}(\Gamma_{0}(N^{2})\ltimes (N\Z \times \Z),\chi)$ be a Jacobi form with integral algebraic coefficients which is an eigenfunction of all Hecke operators $T_{p}$ for $(p,mN) = 1$. Let $S = \{p_{1},\dots,p_{s}\}$ be a finite set of odd primes which are coprime to $mN$ and let $\varepsilon_{1},\dots,\varepsilon_{s} \in \{\pm 1\}$. Let $\ell$ be a prime with $(\ell,2mN) = 1$, $\ell \notin A(p_{j},\lambda_{p_{j}})$ for all $p_{j} \in S$ and $\nu_{\ell}(c(n,r)) = 0$ for some $D = r^{2}-4mn < 0$ with $(D,mN) = 1$. Then there exist infinitely many fundamental discriminants $D = r^{2}-4mn < 0$ prime to $mN$ such that $\big(\frac{D}{p_{j}}\big) = \varepsilon_{j}$ for all $p_{j} \in S$ and $\nu_{\ell}(c(n,r)) = 0$.
	\end{Theorem}

		\begin{proof}
		By Lemma \ref{TwistingProjectionOperators} the projected function
		\[
		\phi' = \phi\bigg|\prod_{p \mid mN}\left(1 - B_{p}\right) = \sum_{\substack{n,r \in \Z \\ r^{2}< 4mn \\ (r,mN) = 1}}c(n,r)q^{n}\zeta^{r}
		\]
		is a Jacobi form for $\Gamma_{0}(m^{2}N^{2})\ltimes (mN\Z \times \Z)$ with integral algebraic Fourier coefficients, and we have $\nu_{\ell}(\phi') = 0$ by assumption. Further, $\phi'$ is still an eigenfunction of all $T_{p}$ with $(p,mN) = 1$.  
		
		A repeated application of Theorem \ref{OneDiscriminant} shows that the function
		\[
	\phi'' = \phi'\bigg|\prod_{j=1}^{s-1}(1-B_{p_{j}}+\varepsilon_{j} (\cdot)_{\psi_{j}}) =\sum_{\substack{n,r \in \Z \\ (r,mN) = 1 \\ (\frac{D}{p_{j}}) = \varepsilon_{j}, 1 \leq j \leq s-1}}c(n,r)q^{n}\zeta^{r}
		\]
		(with $\psi_{j} = \big(\frac{\cdot}{p_{j}}\big)$) again satisfies $\nu_{\ell}(\phi'') = 0$ and is an eigenform of $T_{p}$ for $p$ prime to $mN$ and $p_{1},\dots,p_{s-1}$. 
		
		If we now apply Theorem \ref{InfinitelyDiscriminants} to $\phi''$ with $p = p_{s}$ and $\varepsilon = -\varepsilon_{s}$, we find infinitely many discriminants $D = r^{2}-4mn < 0$ prime to $mN$ with distinct square free parts such that $\big( \frac{D}{p_{j}}\big) = \varepsilon_{j}$ for all $j = 1,\dots,s$ and $\nu_{\ell}(c(n,r)) = 0$. Writing each of these discriminants as $D = f^{2}D_{0}$ with necessarily pairwise distinct fundamental discriminants $D_{0}$, Lemma \ref{FundamentalDiscriminants} shows that $D_{0} = r_{0}^{2}-4mn_{0}$ is a square mod $4m$ with $c(n,r) = c(f^{2}n_{0},fr_{0})$, and the eigenform relations from Lemma \ref{HeckeEigenformRelation} give $\nu_{\ell}(c(n_{0},r_{0})) \leq \nu_{\ell}(c(n,r)) = 0$, i.e. $\nu_{\ell}(c(n_{0},r_{0})) = 0$. This completes the proof.
		\end{proof}

		Now we restrict to $N = 1$. In order to infer Theorem \ref{MainTheorem} from the last theorem, it only remains to show that for a Hecke eigenform $\phi \in J_{k,m}$ (with $k > 2$ if $\phi$ is not a cusp form) and each prime $p$ with $(p,2m) = 1$ the set $A(p,\lambda_{p})$ is finite, or equivalently, $\lambda_{p} \neq \pm (p^{k-1}+p^{k-2})$.

		First suppose that $\phi \in J_{k,m}^{\cusp}$ is a cusp form. By \cite{SkoruppaZagier}, Theorem 5, $J^{\cusp}_{k,m}$ is Hecke-equivariantly isomorphic to a certain subspace of $S_{2k-2}(m)$, so $\lambda_{p}$ is also an eigenvalue of the Hecke operator $T_{p}$ on $S_{2k-2}(m)$. Now \cite{KohnenRemark} gives the estimate $|\lambda_{p}| < p^{k-1} + p^{k-2}$.

	If $0 \neq \phi \in J_{k,m}$ is a non-cuspidal Jacobi form of weight $k > 2$ and an eigenform of all Hecke operators $T_{p}$ for $(p,m) = 1$, then by Theorem 4.4 in \cite{EichlerZagier} the form $\phi$ lies in the space spanned by the $(U_{d}\circ V_{d'})$-images of Eisenstein series $E_{k,m'}^{(\chi)}$ of index $m' = f^{2}$ defined after Theorem 2.4 in \cite{EichlerZagier}, where $\chi$ runs through the primitive Dirichlet characters mod $f$ with $\chi(-1) = (-1)^{k}$. It was shown at the end of the proof of Theorem 2 in \cite{SkoruppaZagier} that $E_{k,m'}^{(\chi)}$ is an eigenform of $T_{p}$ for $(p,m') = 1$ with eigenvalue $\sigma_{2k-3}^{(\chi)}(p) = \sum_{d\mid p}d^{2k-3}\overline{\chi}(d)\chi(p/d)$. Hence the eigenvalues of $\phi$ are also of this form, and for $k > 2$ we have $\sigma_{2k-3}^{(\chi)}(p) \neq \pm(p^{k-1}+p^{k-2})$. 
	
	Thus the finite set of primes $\ell$ with $(\ell,2m) = 1$ for which Theorem \ref{MainTheorem} does not hold is given by the primes in 
	\begin{align}\label{ExceptionalPrimes2}
	\bigcap_{\substack{p \text{ prime} \\ (p,2m) = 1}}A(p,\lambda_{p}) \ \cup \ \{\ell \text{ prime}:\nu_{\ell}(c(n,r)) > 0 \text{ for all $r^{2}\leq 4mn$ with $(r,m) = 1$}\}
	\end{align}

%	Let $\phi \in J_{k,m}^{\new}$ be a Jacobi newform with integral algebraic Fourier coefficients which is an eigenfunction of all Hecke operators $T_{p}$ for $(p,m) = 1$. Then $\phi$ is also an eigenfunction of the Atkin-Lehner involutions $w_{q}$ defined in \cite{EichlerZagier}, where $q \mid \mid m$. Let $\lambda_{p}$ and $\epsilon_{q}$ denote the eigenvalues of $\phi$ under $T_{p}$ and $W_{q}$. Define
%	\[
%	A(p,\lambda_{p}) = \{\ell \text{ prime}: \nu_{\ell}(\lambda_{p}-\chi(p)\varepsilon (p^{k-1}+p^{k-2})) > 0 \text{ for some $\varepsilon  \in\{ \pm 1\}$}\} \cup \{\ell \mid 2mp(p-1)\}.
%	\]
	
	Now let $\phi \in J_{k,m}^{\new}$ be a newform in the sense of \cite{EichlerZagier} which is an eigenform of all Hecke operators $T_{p}$ with $(p,m)=1$. Then $\phi$ is also an eigenform of the Atkin-Lehner involutions $w_{m'}, m' \mid \mid m,$ defined in \cite{EichlerZagier}, Theorem 5.2. Let $\epsilon_{m'} \in \{\pm 1\}$ denote the eigenvalue of $\phi$ under $w_{m'}$. The following proposition shows that if $\phi$ is a newform, we can replace in Theorem \ref{MainTheoremTechnical} the requirement that there be a discriminant $D = r^{2}-4mn < 0$ prime to $m$ with $\nu_{\ell}(c(n,r)) = 0$ by the conditions $\nu_{\ell}(\phi) = 0$ and $\ell \nmid (p-\epsilon_{p})$ for all primes $p \mid \mid m$.
	
	\begin{Proposition}
		Let $\phi \in J_{k,m}^{\new}$ be a Jacobi newform with integral algebraic coefficients which is an eigenfunction of all Hecke operators $T_{p}$ for $(p,m) = 1$. Let $\ell$ be a prime with $(\ell,m) = 1$ and $\nu_{\ell}(\phi) = 0$ such that $\ell$ does not divide $p-\epsilon_{p}$ for any prime $p \mid \mid m$, where $\epsilon_{p} \in \{\pm 1\}$ denotes the eigenvalue of $\phi$ under the Atkin-Lehner involution $w_{p}$. Then there exists a discriminant $D = r^{2}-4mn < 0$ prime to $m$ such that $\nu_{\ell}(c(n,r)) = 0$.
	\end{Proposition}
	
	\begin{proof}
		The proof goes along the same lines as the proof of Lemma 3.1 in \cite{SkoruppaZagier}: Assume that $\nu_{\ell}(c(n,r)) > 0$ for all $r^{2} <4mn$ with $(r,m) = 1$. Using the formula $\sum_{t \mid m}\mu(m)\prod_{p\mid t}e(r/p) = 0$ for $(r,m) > 1$ the assumption gives
		\[
		\nu_{\ell}\left( \sum_{t \mid m}\mu(t)\prod_{p \mid t}\phi\bigg|_{m}\left[0,\frac{g}{p}\right]\right) > 0
		\]
		for every integer $g$. Let $f(\tau,z)$ denote the function in the braces. It is easily checked that $f$ is a Jacobi form for the group $\Gamma_{1}(m^{2}) \ltimes (m\Z \times \Z)$ with integral algebraic Fourier coefficients. Theorem \ref{qExpansionPrinciple} shows that for $M \in \SL_{2}(\Z)$ the function $f|_{k,m}M$ has algebraic Fourier coefficients and $\nu_{\ell}(f|_{k,m}M) \geq \nu_{\ell}(f) > 0$. Now the calculation in the proof of Lemma 3.1 in \cite{SkoruppaZagier} gives
		\[
		\nu_{\ell}\left(\phi\bigg|\left(\prod_{d^{2}\mid m}\frac{\mu(d)}{d}u_{d}U_{d}\right)\circ\left( \prod_{\substack{p \mid \mid m }}\left(1 - \frac{1}{p}w_{p}\right)\right) \right) > 0
		\]
		where $u_{d}$ is the auxiliary operator from \cite{SkoruppaZagier}, line (6), and $\phi|U_{d} = \phi(\tau,dz)$. Since $\phi$ is a newform and an eigenform of all Hecke operators, we have $\phi|u_{d} = 0$ for all $d^{2}|m$ with $d > 1$. In particular, we find
		\[
		\nu_{\ell}\left(\phi\bigg|\prod_{\substack{p \mid \mid m  }}\left(1 - \frac{1}{p}w_{p}\right)\right) > 0
		\]
		The inner function is just $\prod_{p \mid \mid m}(1-\epsilon_{p}/p)\phi$, so we get $\nu_{\ell}(\phi) > 0$, contradicting the assumption $\nu_{\ell}(\phi) = 0$. Hence there is some $D =r^{2}-4mn < 0$ prime to $m$ with $\nu_{\ell}(c(n,r)) = 0$.
	\end{proof}

%	For the proof of Theorem ... we construct from $\phi = \sum_{n,r}c(n,r)q^{n}\zeta^{r} \in J_{k,m}^{\cusp,\new}$ the function
%	\[
%	\phi' = \sum_{\substack{n,r \in \Z \\ (r,m) = 1 \\ (\frac{D}{p_{j}}) = \varepsilon_{j}, 1 \leq j \leq s}}c(n,r)q^{n}\zeta^{r}
%	\]
%	which is a cusp form for $\Gamma_{0}(N^{2})\ltimes (N\Z \times \Z)$ with $\nu_{\ell}(\phi') = 0$, where $N = m p_{1}\cdots p_{s}$. By Lemma ... applied to $\phi'$ we find infinitely many discriminants $D = r^{2}-4mn < 0$ with pairwise distinct square-free parts and $(r,m) = 1$ such that $\nu_{\ell}(c(n,r)) = 0$ and $\big( \frac{D}{p_{j}}\big) = \varepsilon_{j}$ for all $p_{j} \in S$. If $D = f^{2}D_{0}$ for a fundamental discriminant $D_{0} < 0$, then $D_{0} = r_{0}^{2}-4mn_{0}$ is a square mod $4m$ and $c(n,r) = c(f^{2}n_{0},fr_{0})$ by Lemma... Finally, Lemma ... gives $\nu_{\ell}(c(n_{0},r_{0})) \leq \nu_{\ell}(c(n,r)) = 0$. Since the square-free parts of the discriminants $D$ are different, the corresponding fundamental discriminants are different.
%	
%
	\bibliography{references}{}
\bibliographystyle{alphadin}

\end{document}